\documentclass[11pt]{article}

\usepackage{tikz}
\usepackage{times,amsfonts,amsmath,amstext,amsbsy,amssymb,
  amsopn,amsthm,upref,eucal,amscd}
\usepackage[colorlinks=true]{hyperref}
\usepackage[T1]{fontenc}
\usepackage[backgroundcolor=white]{todonotes}
\usepackage{multirow}
\usepackage{algorithm}
\usepackage{enumerate}
\usepackage[noend]{algpseudocode}
\usepackage{color}
\newtheorem{theorem}{Theorem}

\newtheorem{proposition}{Proposition}
\newtheorem{corollary}[proposition]{Corollary}
\newtheorem{lemma}[proposition]{Lemma}
\newtheorem{remark}[proposition]{Remark}
\newtheorem{definition}[proposition]{Definition}
\newtheorem{example}[proposition]{Example}

\numberwithin{proposition}{section}
\numberwithin{equation}{section}

\makeatletter
\newcommand{\subjclass}[2][1991]{%
  \let\@oldtitle\@title%
  \gdef\@title{\@oldtitle\footnotetext{#1 \emph{Mathematics subject classification.} #2}}%
}

\begin{document}

\title{Effective estimates of ergodic quantities illustrated on the Bolyai-R\'enyi map}

\author{Mark Pollicott and Julia Slipantschuk\footnote{The authors were partly supported by ERC-Advanced Grant 833802-Resonances.}}




\newcommand{\Addresses}{{
  \bigskip
  \footnotesize

  M.~Pollicott, \text{Department of Mathematics, University of Warwick, Coventry, CV4 7AL, UK.}\par\nopagebreak
  \textit{E-mail}: \texttt{masdbl@warwick.ac.uk}

  \medskip

  J.~Slipantschuk, \text{Department of Mathematics, University of Warwick, Coventry, CV4 7AL, UK.}\par\nopagebreak
  \textit{E-mail}: \texttt{julia.slipantschuk@warwick.ac.uk}
}}

\maketitle

\abstract{We present a practical and effective method
for rigorously estimating quantities associated to top
eigenvalues of transfer operators to very high precision. The method combines
explicit error bounds of the Lagrange-Chebyshev approximation with
an established min-max
method.
We illustrate its applicability by significantly improving rigorous estimates on
various ergodic quantities associated to the Bolyai-R\'enyi map.}

\section{Introduction}

In a classic 1957 paper, A. R\'enyi described a general class of $f$-expansions.    This included, as special cases, base $d$-expansions and $\beta$-expansions; continued fraction expansions; and (as Example 3 in \cite{renyi}) the iterated radical expansions, which first appeared in Bolyai's 1832 book \cite{bolya}.   This last example
gives an expansion for almost every  $0 < x < 1$ in the form of an iterated square root:
\begin{equation}\label{eq:renyi_exp}
x = -1 
+ \sqrt{a_1 + \sqrt{a_2 + \sqrt{a_3 + \cdots}}},
\end{equation}
where $a_1, a_2, a_3, \ldots  \in \{1,2,3\}$.
Related infinite radical expansions appear in the notebooks of Ramanujan \cite{ramanujan}. 

Statistical properties of the expansions in \eqref{eq:renyi_exp} are given
by ergodic properties of the associated Bolyai-R\'enyi transformation $T\colon[0,1) \to [0,1)$ - a piecewise analytic expanding Markov map (see Section \ref{sec:BR}) with a unique $T$-invariant probablity
measure $\mu$. Bosma, Dajani and Kraaikamp \cite{bdk} conjectured
that the metric entropy is $h(T,\mu)\approx 1.0545$, accurate to two decimal places.
In \cite{jp} a rigorous estimate of $1.056313074$ was given using a periodic point method.
One of the aims of this note is to use estimation of this and other quantities
to illustrate an alternative approach leading to significantly improved estimates.

In this and similar systems, estimating quantities of interest such as
entropy, or digit frequencies, requires approximation
of the pressure function arising from the maximal positive
eigenvalue of an appropriately chosen (positive) transfer operator $\mathcal{L}$.
This eigenvalue satisfies the so-called min-max principle,
allowing to bound the pressure function in terms of bounds on
$(\mathcal{L}f) / f$ for suitable positive functions $f$, see \cite[Theorems 7.25-7.26]{mayer}
and \cite{PV1}.

Two crucial ingredients are needed to make these bounds rigorous and sufficiently tight:
the first consists in choosing a suitable $f$, the
second in obtaining rigorous and tight lower and upper bounds on $(\mathcal{L}f) / f$.
For the first step, in order to obtain effective bounds on the pressure function,
the function $f$ should be chosen close to the leading eigenfunction of $\mathcal{L}$.
This can be achieved by approximating
$\mathcal L$ by a suitable finite-rank operator $\mathcal L_n$ and using its leading
eigenfunction as a candidate for $f$. In \cite{PV1, PS} this was accomplished using the classical
Lagrange-Chebyshev interpolation scheme, for which results in \cite{BS} guarantee exponential
convergence (in rank $n$) of eigenfunctions of $\mathcal L_n$ to those of $\mathcal{L}$ for
analytic expanding maps $T$.
For the second step, rigorous lower and upper bounds on $(\mathcal Lf) / f$ were obtained
using ball arithmetic, involving the use of higher-order derivatives
for increased precision.

In this note we provide a refinement of this method, critically also using the
Lagrange-Chebyshev interpolation method for enhancing the second step.
In particular, the exponential convergence properties of this scheme allow
us to rigorously bound $(\mathcal{L} f) / f$ to much greater precision,
while being more efficient and using a single derivative of the quotient function.
Overall, leveraging the strong convergence properties of Lagrange-Chebyshev interpolation
for efficiently achieving accurate approximation, and ball arithmetic for rigorous bounding, this
``Lagrange-Chebyshev min-max algorithm''
provides a method for rigorous high-accuracy estimation of ergodic quantities for analytic expanding maps,
at comparatively small computational cost.

This note is organized as follows. We introduce the Bolyai-R\'enyi map
and state our approximation results for several of its ergodic quantities in \S\ref{sec:BR},
and summarize the well-known properties of the associated
transfer operator and its leading eigenvalue in \S\ref{sec:transfer_op}.
In \S\ref{sec:algo_section} we describe our min-max method:
we derive the key estimate Theorem \ref{thm:bounds_h} in \S\ref{sec:minmax},
describe the choice of approximate eigenfunctions $f$ in \S\ref{sec:choosing_fg}
and the enhanced method for bounding $(\mathcal{L}f) / f$ in \S\ref{sec:rigorous_ineq}, and
provide pseudo-code for the method's implementation in \S\ref{sec:algorithm_impl}.
Proofs of the results from \S\ref{sec:BR} based on this method are given in \S\ref{sec:proofs}.
Finally, in \S\ref{sec:comparison} we compare our algorithm with other available methods,
and finish off with some generalizing remarks in \S\ref{sec:remarks}.

\section{Bolyai-R\'enyi transformation and results}\label{sec:BR}

Associated to \eqref{eq:renyi_exp} is the Bolyai-R\'enyi transformation $T\colon[0,1) \to [0,1)$ of the form
$$
T(x)
=
\begin{cases}
x^2 + 2x &\hbox{ if } 0 \leq x < \sqrt{2} -1 \cr
x^2 + 2x -1&\hbox{ if } \sqrt{2} -1 \leq x < \sqrt{3} -1 \cr
x^2 + 2x -2&\hbox{ if } \sqrt{3} -1 \leq x \leq 1. \cr
\end{cases}
$$
We can then write digits in the radical expansion \eqref{eq:renyi_exp} as
$$
a_n
=
\begin{cases}
1 &\hbox{ if } 0 \leq T^{n-1}(x) < \sqrt{2} -1 \cr
2 &\hbox{ if } \sqrt{2} -1 \leq T^{n-1}(x) < \sqrt{3} -1 \cr
3 &\hbox{ if } \sqrt{3} -1 \leq T^{n-1}(x) \leq 1. \cr
\end{cases}
$$

Formally, it is often convenient to consider $T$ as a piecewise expanding Markov map acting on three closed intervals.
 Since $T$ is expanding (with $\inf_x T'(x) \geq 2$) there is a (unique) $T$-invariant probability measure $\mu$
 which is absolutely continuous (even equivalent) to Lebesgue measure by, for example, the classical Lasota-Yorke theorem \cite{ly}.
We denote the metric entropy of $\mu$ by $h(\mu) = h(T, \mu)$, and by the classical Rokhlin formula
it is also conveniently  given as the Lyapunov exponent, i.e.,
 \begin{equation}\label{eq:entropy}
 h(\mu) = \int \log T'(x) d\mu(x),
\end{equation}
 which, as usual, quantifies the sensitive dependence on initial conditions for typical orbits.

\begin{theorem}\label{thm:entropy}
The metric entropy of $T$ can be rigorously estimated to be
\begin{equation*}
h(\mu) = 1.0563130740\,7297055209\,9568877064\,0651679335\,4262184005 \ldots,
\end{equation*}
accurate to the number of decimal places presented.
\end{theorem}

The min-max method we use to obtain this new estimate is relatively undemanding on computer resources and so with more patience it would be relatively easy to get more decimal places.

\begin{example}
As an application we have a variant on Lochs' theorem, which originally compared base $10$ expansions with
continued fraction expansions of typical digits.
For a.e.~($\mu$) point $x\in (0,1)$ and $n \geq 1$ we can write
$$
\begin{aligned}
x &=  -1 + \sqrt{a_1 + \sqrt{a_2 + \sqrt{a_3 + \cdots}}}
&\text{ where }& a_1, a_2, a_3, \ldots \in \{1,2,3\}   \cr
&= \frac{b_1}{3} + \frac{b_2}{3^2} + \frac{b_3}{3^3} + \cdots
&\text{ where }& b_1, b_2, b_3, \ldots \in \{0,1,2\}   \cr
\end{aligned}
$$
and for each $n \geq 1$ we denote by $N_n(x)$ the minimal number of digits
($b_1, b_2, \ldots, b_{N_n(x)}$) in base $3$ expansion which are needed to
specify the first $n$ digits ($a_1, a_2, \ldots, a_n$) in the radical expansion.
The following is a simple consequence of Theorem \ref{thm:entropy}.

\begin{corollary}
The quantity $\lim_{n\to +\infty} (N_n/n)$ evaluates to
$$1.0400442024\,5603917587\,6064437681\,7446547515\,185552379 \ldots,$$
accurate to the number of decimal places presented.
\end{corollary}

\begin{proof}
The proof of Lochs' theorem is based on the Shannon-McMillan-Breman theorem \cite{lochs}, generalized to the present example to give that the limit exists and is equal to the ratio of the entropies of the two maps, that is, $\log(3) / h(\mu) $,
which corresponds to the value stated.
\end{proof}

\end{example}

In the following two subsections we will provide further applications
of our method, presenting estimates for several other quantities
related to the Bolyai-R\'enyi map.

\subsection{Frequency of digits}
We consider the frequency  of digits that occur in typical expansions.  In particular,
we can define 
$$
f_i (x):= \lim_{n \to +\infty}
\frac{1}{n}
\#\{1 \leq k \leq n : a_k(x) = i\} \quad \text{ for } i = 1,2,3
$$
whenever  the limit exists. By ergodicity of $T$ the limit exists and is constant for a.e.~($\mu$)
$x \in [0,1)$ and we simply denote it $f_i$.

\begin{theorem} \label{thm:freqs} The digit frequencies are
\begin{align*}
f_1 &= 0.4640796294\,4716719166\,0214542662\,4296264246\,0872990983\ldots, \\
f_2 &= 0.3044190449\,4046044774\,3959549801\,4270582974\,2520632111\ldots, \\
f_3 &= 0.2315013256\,1237236059\,5825907536\,1433152779\,6606376905\ldots,
\end{align*}
accurate to the number of decimal places presented.
\end{theorem}

These are an improvement on the accuracy of previous estimates in \cite[Theorem 2]{jp},
which were computed up to approximately $\pm 10^{-7}$.

\subsection{Deleted digits}
Another interesting problem is to consider the Cantor set whose radical expansions use only the digits $1$ and $3$, i.e., the ``deleted digit'' Cantor set 
$$
\mathcal E = \left\{
x = -1 
+ \sqrt{a_1 + \sqrt{a_2 + \sqrt{a_3 + \cdots}}} : a_1, a_2, a_3, \ldots \in \{1,3\}
\right\}.
$$

\begin{theorem} \label{thm:hdim}
The Hausdorff dimension of $\mathcal E$ is
\begin{align*}
\dim_H(\mathcal E) =
0.6439131204\,7072945768\,7895134656\,7617073899\,0093573261\ldots,
\end{align*}
accurate to the number of decimal places presented.
\end{theorem}

Similar problems have attracted interest in the context of continued fractions
\cite{good,PV1} and this result might be considered as an analogy.

\section{The transfer operators \texorpdfstring{$\mathcal L_t$}{}}\label{sec:transfer_op}

We associate to $T$ the three  inverse branches $T_1, T_2, T_3\colon [0, 1] \to [0, 1]$ given by
$$
T_i(x) = \sqrt{i + x} - 1 \quad \text{ for } i=1,2,3
$$
(i.e., $T\circ T_i(x) = x$ for all $x\in [0,1)$ and $i=1,2,3$)
and observe by direct computation that 
$$
T_i'(x) =  \frac{1}{2(T_i(x) + 1)} \quad \text{ for } i=1,2,3.
$$

Let $C([0,1])$ denote the Banach space of continuous functions $f\colon[0,1] \to \mathbb R$
with the supremum norm $\|f\|_\infty = \sup_{x\in [0,1]}|f(x)|$.
We recall the following well-known result.

\begin{lemma}[Lasota-Yorke \cite{ly}]
The absolutely continuous $T$-invariant  measure $\mu$ on $[0,1]$  satisfies 
$h(x) := \frac{d\mu}{dx} \in C([0,1])$, where $h$ is the eigenvector associated to the maximal eigenvalue $1$
for the linear operator $\mathcal L_0 \colon C([0,1]) \to C([0,1])$ given by
$$
\mathcal L_0 f(x) = 
\sum_{i=1}^3|T_i'(x)| f(T_ix),
\quad \forall f \in C([0,1]), x \in [0,1].
$$
\end{lemma}

More generally,  we can consider the  following parameterized
{\it family} of bounded linear operators.

\begin{definition}
For each  $t\in \mathbb R$ we 
let 
 $\mathcal L_t \colon C([0,1])\to C([0,1])$ denote the bounded linear operator
 given by
$$
\mathcal L_t f(x) =
\sum_{i=1}^3|T_i'(x)|^{1+t} f(T_ix),
\quad \forall f \in C([0,1]), x \in [0,1].
 $$
\end{definition}

We can denote the spectral radius of $\mathcal L_t$ by $e^{P(t)}$.   The following formulation follows from the
spectral radius theorem.

\begin{definition}
We can associate the \emph{pressure function} $P\colon \mathbb R \to \mathbb  R$ defined by
$$
P(t) =  \lim_{n\to +\infty} \frac{1}{n} \log \|\mathcal L_t ^n1\|_\infty,
$$
where $1(x) = 1$ is the constant function.
\end{definition}

As will be seen from the next lemma, the limit exists, and the pressure function
defined this way has certain properties following from the transformation $T$
being expanding and differentiable. We note that there are alternative
definitions of pressure, e.g.~via the variational principle applied to the function $-t \log |T'|$, see \cite{walters}.

\begin{lemma} \label{lem:pressure}
The function $P\colon \mathbb R \to \mathbb R$ has the following properties:
\begin{enumerate}
\item $P(0) = 0$. 
\item 
$P$ is $C^1$ (even real analytic), monotone decreasing and convex.
\item $\frac{d}{dt}P(t)|_{t=0} = - h(\mu)$.
\end{enumerate}
\end{lemma}

\begin{proof}
The first part follows from bounds on $\| \mathcal L_0^n 1\|_\infty$, for $n \geq 0$.

For the second part, the  value  $e^{P(t)}$ for   the spectral radius of
$\mathcal L_t$ occurs as a maximal simple eigenvalue for the transfer operator
$\mathcal L_t\colon C^1([0, 1]) \to C^1([0, 1])$
 acting on the Banach space of $C^1$ functions $f\colon [0, 1] \to \mathbb R$ with the norm $\|f\| := \|f\|_\infty +  \|f'\|_\infty$
 (Ruelle Operator  Theorem, see \cite{pp}).
 The analyticity of $P(t)$ then follows by perturbation theory  applied to the maximal eigenvalue of $\mathcal L_t$, see \cite{pp, ruelle}.
 The monotone decreasing property follows from the negative derivative in the next part.
 Convexity follows from the observation that for each $n \geq 1$ the function $t \mapsto \|\mathcal L_t^n 1\|_\infty$ is convex.
 
For the third part, let $h_t \in C^1([0, 1])$ be an eigenfunction corresponding
to the maximal eigenvalue of $\mathcal L_t$ for $t \in \mathbb R$,
and write $h'_t(x) = \frac{d}{dt}h_t(x)$.
By analytic perturbation theory \cite{kato},
we can differentiate the eigenvalue equation
$\mathcal L_t h_t =  e^{P(t)} h_t$ at $t=0$ to write
  $$\mathcal L_0h_0' - \mathcal L_0(\log|T'| h_0)  =  e^{P(0)}h_0' + P'(0) h_0.$$  We can then integrate with respect to the measure $\nu_0 = \mathcal L_0^*\nu_0$ to write
  \begin{align}\label{eq:pressure_der}
  P'(0) = -  \nu_0(\log|T'| h_0) = - \mu(\log|T'|).
  \end{align}
  Finally, the Rokhlin identity gives $h(\mu) = \mu(\log|T'|)$.
\end{proof}

We will make use of part 3 of Lemma \ref{lem:pressure}, despite the unpromising formulation,
to  bound the entropy $h(\mu)$ by deriving inequalities on the  derivative of $P(t)$ at $t=0$.
In particular, using elementary calculus,
we can estimate $\frac{d}{dt}P(t)|_{t=0} = -h(\mu)$
from estimates of the two values of pressure $P(\epsilon)$ and $P(-\epsilon)$
for a fixed $\epsilon > 0$.

\begin{figure}
\begin{center}
\begin{tikzpicture}[thick,scale=0.8, every node/.style={scale=0.8}]
\draw[->] (0,0) --(10,0);
\draw[dashed, <->] (7.2,0.2)--(4,0.2);
\draw[dashed, <->] (1,-0.2)--(4,-0.2);
\draw[dashed, <->] (1,2.0)--(1,0);
\draw[dashed, <->] (7.2,-0.8)--(7.2,0);
\draw[thick] (1.3,1.1)--(7.2,-1.2);
\draw[thick,dashed] (1,2.0)--(4.0,0);
\draw[thick,dashed] (4.0,0)-- (7.2,-0.8); 
\node at (0.1,  1.0) {$P(-\epsilon)$};
\node at (8.2,  -0.4) {$|P(\epsilon)|$};
\node at (2.5,  -0.5) {$\epsilon$};
\node at (5.7,  0.5) {$\epsilon$};
\draw[red](-1,5) .. controls (1,1)  and (3,0) .. (8,-1.0);
\node at (4.3,  -0.5) {$0$};
\node at (0.5,  3.9) {$P(t)$};
\node at (9.5,  0.5) {$t$};
\node at (6,  -1.4) {slope $-h(\mu)$};
\end{tikzpicture}
\end{center}
\caption{The pressure function $P(t)$ and the inequalities in Corollary \ref{cor:entropy_bounds}.}
\label{fig:pressure}
\end{figure}
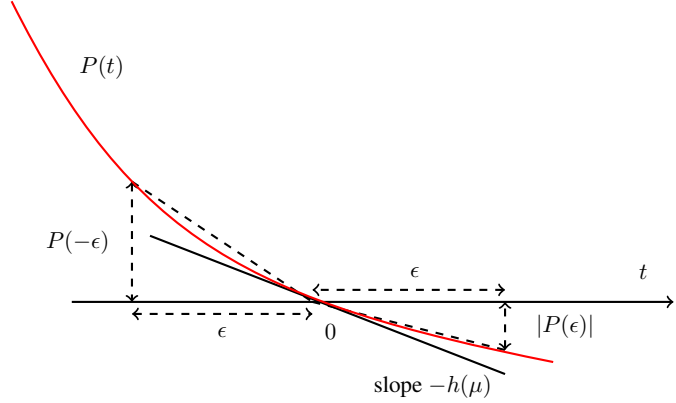

By convexity we can (rigorously) bound the entropy (cf. Figure \ref{fig:pressure}):
\begin{corollary} \label{cor:entropy_bounds}
For any $\epsilon > 0$ we have the bounds
\begin{equation} \label{eq:h_eps_bounds}
\frac{P(-\epsilon)}{\epsilon}
\geq  h(\mu) \geq 
\frac{|P(\epsilon)|}{\epsilon}.
\end{equation}
\end{corollary}

\begin{remark}
To make the bounds in \eqref{eq:h_eps_bounds} effective, we need to choose $\epsilon > 0$ appropriately small.
On the other hand, for large values of $\epsilon$ we observe
that the bounds become the asymptotic derivatives $P'(t)$
of the pressure $P(t)$ as $t \to \pm \infty$. Interestingly these become
$$
\inf_{\nu} \nu(\log |T'|) \leq h(\mu) \leq \sup_{\nu} \nu(\log |T'|),
$$
where the supremum and infimum range over all $T$-invariant probability measures,
although these are clearly not useful for practical estimates.
\end{remark}

\section{Algorithm for estimating properties of \texorpdfstring{$P(\pm \epsilon)$}{the pressure function}}\label{sec:algo_section}
\label{sec:method}

The bounds in \eqref{eq:h_eps_bounds} may at first sight seem unpromising
since we need to choose $\epsilon > 0$ appropriately small.
Therefore it is important to find an efficient and accurate method to estimate $P(\pm\epsilon)$,
or equivalently the leading eigenvalue of $\mathcal{L}_{\pm \epsilon}$.
We shall now describe a simple min-max method to give an upper
bound on this eigenvalue.

\subsection{Min-max method to bound the top eigenvalue of \texorpdfstring{$\mathcal{L}_{\pm \epsilon}$}{the transfer operator}}
\label{sec:minmax}

We will use a ``min-max'' estimate which is based on the following lemma.

\begin{lemma} \label{lem:pe_bound}
Let $t, \alpha \in \mathbb{R}$, and assume there exists a positive continuous
function $g\colon[0,1]  \to \mathbb R^+$ such that
$\sup_x \frac{\mathcal L_{t} g(x)}{g(x)} \leq  e^\alpha$. Then $P(t) \leq \alpha$.
\end{lemma}

\begin{proof}
Since the transfer operator $\mathcal L_t$ is positive, applying it multiple times
to the inequality in the hypothesis yields
$$
0 <
\mathcal L_t^n g(x) \leq e^\alpha \mathcal L_t^{n-1} g(x)
\leq e^{2\alpha} \mathcal L_t^{n-2} g(x)
\leq
\cdots
\leq e^{n\alpha}g(x)
$$
for all $x\in [0, 1]$. In particular, since $g$ is positive we have
$$
e^{P(t)} = \lim_{n \to +\infty} \|\mathcal L_t^n g \|_\infty^{1/n} \leq e^\alpha,
$$
using our definition of $P(t)$. The claim follows.
\end{proof}

Using the fact that $P(t)$ is positive for $t < 0$ and negative for $t > 0$,
the lemma yields the following implications for any $\epsilon, \alpha, \beta > 0$
and any positive continuous functions $f, g$ (cf.~Figure \ref{fig:hypotheses}):
\begin{align}
\sup_x \frac{\mathcal L_{\epsilon} g(x)}{g(x)} \leq  e^{-\alpha}
&\implies |P(\epsilon)| \geq \alpha, \label{eq:pe_lower_bound} \\
\sup_x \frac{\mathcal L_{-\epsilon} f(x)}{f(x)} \leq e^\beta
&\implies P(-\epsilon) \leq  \beta. \label{eq:pe_upper_bound}
\end{align}
Combining with the inequalities in \eqref{eq:h_eps_bounds}, we have the following useful bound.

\begin{figure}
\centerline{
      \begin{tikzpicture}[thick,scale=0.6, every node/.style={scale=0.5}]
        \draw[<-] (0,5) --(0,0); 
        \draw[-] (-0.5,0) --(8,0);
        \draw[red](0,3) .. controls (2,4)  and (6,2) .. (8,3); 
        \node at (0,  -0.5) {$0$};
           \draw[blue](0,2) .. controls (2,3)  and (6,1) .. (8,2); 
        \node at (0,  -0.5) {$0$};
                \node at (8,  -0.5) {$1$};
        \node at (8.5,  2.0) {$\mathcal L_{\epsilon}g$};
            \node at (8.5,  3.2) {$e^{-\alpha}g$};
    \end{tikzpicture}
    \hskip 1cm
        \begin{tikzpicture}[thick,scale=0.6, every node/.style={scale=0.5}]
        \draw[<-] (0,5) --(0,0); 
        \draw[-] (-0.5,0) --(8,0);
        \draw[blue](0,3) .. controls (2,1)  and (6,4) .. (8,3); 
        \node at (0,  -0.5) {$0$};
           \draw[red](0,4) .. controls (2,3)  and (6,6) .. (8,4); 
        \node at (0,  -0.5) {$0$};
                \node at (8,  -0.5) {$1$};
        \node at (8.5,  4.0) {$ e^\beta$ f};
            \node at (8.5,  2.7) {$\mathcal L_{-\epsilon}f$};
    \end{tikzpicture}
}    
\caption{Illustrations of the hypotheses of \eqref{eq:pe_lower_bound} (left)
and \eqref{eq:pe_upper_bound} (right).}
\label{fig:hypotheses}
\end{figure}
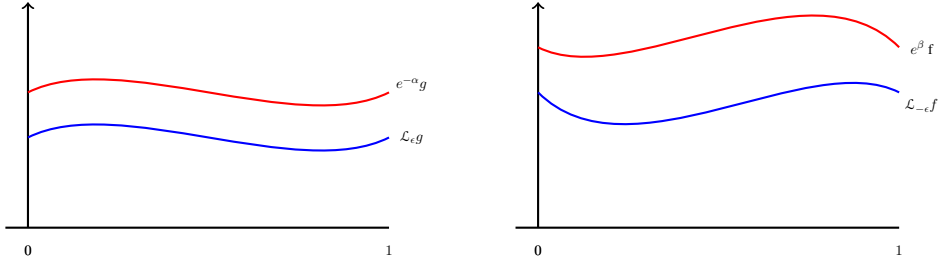

\begin{theorem}[Bounds on the entropy $h(\mu)$] \label{thm:bounds_h}
Let $\epsilon > 0$, and assume there are positive continuous functions
$g, f \colon [0, 1] \to \mathbb R^+$ and $\alpha, \beta > 0$
satisfying the hypotheses of \eqref{eq:pe_lower_bound} and \eqref{eq:pe_upper_bound}, respectively, then
\begin{equation}\label{eq:entropy_bounds}
\frac{\beta}{\epsilon}
\geq h(\mu)  \geq 
\frac{\alpha}{\epsilon}.
\end{equation}

\end{theorem}

In practice, dividing by a small value $\epsilon$ needs to be more than compensated for by the efficiency of the approach in estimating $P(-\epsilon)$ and $P(\epsilon)$, to allow $\alpha$ and
 $\beta$ to be chosen very close (compared to $\epsilon >0$).
There are two crucial ingredients to make the bounds \eqref{eq:entropy_bounds}
effective, rigorous and computationally efficient.
The first consists in choosing $g, f$ in
\eqref{eq:pe_lower_bound}, \eqref{eq:pe_upper_bound} as close as possible
to the leading eigenfunctions of $\mathcal{L}_{\pm \epsilon}$, the second
in efficiently computing tight rigorous bounds on $(\mathcal L g) / g$ and $(\mathcal L f) / f$.
We now describe effective schemes for both, based on Lagrange-Chebyshev interpolation.

\subsection{Choosing the functions \texorpdfstring{$f$ and $g$}{f and g} via interpolation} \label{sec:choosing_fg}

Given $\epsilon > 0$, our aim is to
choose  suitable functions  $f, g\colon [0,1] \to \mathbb R^+$ for which
we can associate values $\alpha,\beta>0$ satisfying \eqref{eq:pe_lower_bound} and \eqref{eq:pe_upper_bound}
and then apply  the bounds in \eqref{eq:entropy_bounds}.
Although there is plenty of scope in how they
are actually arrived at, the most promising way is to choose $f,g$ as close as possible
to the leading eigenfunctions of $\mathcal{L}_{-\epsilon}$ and $\mathcal{L}_{\epsilon}$, respectively.
As these eigenfunctions are not known explicitly, we approximate them by 
approximating the respective transfer operators by finite-rank operators.

More precisely, denoting by $L\colon B \to B$ any bounded operator on a Banach space $B$,
we can approximate it by a finite-rank operator $P_m L P_m$ where
$P_m\colon B \to B$ is a bounded projection operator of finite rank given by
\[P_m f = \sum_{l=0}^{m-1} e_l^*(f) e_l,\]
with $e_l \in B$ and $e_l^* \in B^*$ for $l=0, \ldots, m-1$. The operator $P_m L P_m$ is represented by an $m \times m$ matrix $M$,
whose $(k,l)$-entry is given by
$$
M_{kl} = e_k^*(L e_l). 
$$

\begin{definition}\label{defn:lagcheb}
For $B = C([-1,1])$, the Lagrange-Chebyshev interpolation scheme consists in using $P_m$ as defined above with
\begin{enumerate}
\item
$e_l$ ($0 \leq l \leq m-1$) the Chebyshev polynomial of the first kind and degree $l$; and
\item
$e_k^*$ ($0 \leq k \leq m-1$) the linear functional given by
$$
e_k^*(f) = \frac{2-\delta_{0,k}}{m} \sum_{j=0}^{m-1} f(x_j) e_k(x_j ),
$$
\end{enumerate}
with $x_0, \ldots, x_{m-1}$ the Chebyshev nodes of order $m$, $x_j = \cos((2j+1)\pi/(2m))$.

For an arbitrary interval $I=[c-r, c+r] = \gamma([-1,1])$, with $\gamma(x) = \gamma_{r,c}(x) = rx+c$ for some $c\in \mathbb{R}$, $r>0$,
we denote  $P_{m, \gamma}$ the corresponding Lagrange-Chebyshev interpolation operator
on $B=C(I)$, defined as above with
$e_l$ and $x_j$ replaced with $\hat{e}_l = e_l \circ \gamma^{-1}$ and $\hat{x}_j =\gamma(x_j)$, respectively.
\end{definition}

We apply the above scheme to our transfer operator $L = \mathcal L_\epsilon\colon C([0,1]) \to C([0,1])$,
using the change of coordinates $\gamma = \gamma_{r,c} \colon [-1, 1] \to [0, 1]$ with $r = c = 1/2$.
We obtain an $m\times m$ matrix $M$, whose right eigenvector $v = (v_0, \ldots, v_{m-1})^T$ corresponding to the largest
eigenvalue of $M$ yields a candidate function $g := g^{(m)} = \sum_{l=0}^{m-1} v_l \hat{e}_l$ to use in \eqref{eq:pe_lower_bound}.
A priori it may not be obvious why $g$ is positive and a good candidate function, however we observe that
$\mathcal{L}_\epsilon$ is not an arbitrary bounded operator. The contractions $T_i$ involved in the definition of $\mathcal L_\epsilon$
are analytic on $[0,1]$, hence extend holomorphically to a complex neighbourhood, in particular to a suitable ellipse $E_{\gamma,R} = \gamma(E_R) \supset [0,1]$, where $E_R \supset [-1, 1]$
is a standard Bernstein ellipse\footnote{A Bernstein ellipse is an ellipse $E_R \subset \mathbb{C}$
with foci $-1$ and $1$ and lengths of major and minor semi-axes given by
$a=\cosh(\log R)$ and $b = \sinh(\log R)$, respectively.} for some $R>0$.
Now, Theorem 3.3 and Corollary 3 of \cite{BS} guarantee that the
(generalized) eigenfunctions of $P_{m, \gamma} \mathcal L_\epsilon P_{m, \gamma}$
converge (in supremum norm) exponentially fast in $m$ to those of $\mathcal{L_\epsilon}$. In particular, for large enough
$m$, the function $g=g^{(m)}$ is positive on $[0,1]$ and approximates 
the eigenfunction for the maximal positive eigenvalue $e^{P(\epsilon)}$ of $\mathcal{L}_\epsilon$.

\subsection{Rigorous bounds on the supremum/infimum} \label{sec:rigorous_ineq}

The main step in establishing the left-hand sides of
\eqref{eq:pe_lower_bound} and \eqref{eq:pe_upper_bound}
comes down to rigorously estimating $\sup_{x\in I} h(x)$
for $h(x) = \mathcal{(L_\epsilon} f)(x) / f(x)$ and $I=[0,1]=[c-r, c+r]$ with $c = r = 1/2$.
By the mean value theorem we have
\begin{equation}\label{eq:derivative_estimate}
\sup_{x\in I} h(x) \leq h(c) + \sup_{x\in I} | h'(x) | \cdot r,
\end{equation}
and using the quotient rule we can write $h' = \psi / f^2$ with
\[
\psi = (\mathcal{L}_\epsilon f)' \cdot f - f' \cdot (\mathcal{L}_\epsilon f),
\]
which we use to rewrite \eqref{eq:derivative_estimate} as
\begin{equation}\label{eq:derivative_estimate_2}
\sup_{x\in I} h(x) \leq h(c) +
\frac{r}{\inf_{x \in I} f(x)^2} \cdot \sup_{x\in I} | \psi(x) |.
\end{equation}
To obtain a tight bound on $\sup_{x\in I} | \psi(x) |$,
we can approximate $\psi$ using the
Langrange-Chebyshev interpolation projector $P_{n, \gamma}$, $n \in \mathbb N$,
introduced in the previous section, and use the fact that
\begin{equation}\label{eq:two_parts}
\sup_{x\in I} |\psi(x)| \leq
\sup_{x\in I} |(\psi-P_{n, \gamma}\psi)(x)| +
\sup_{x\in I} |(P_{n, \gamma}\psi)(x)|.
\end{equation}
We note that the function $\psi$
is holomorphic on $E_{\gamma, R} = \gamma(E_R)$ for a suitable
$R > 1$, whenever $f$ is a polynomial (or more generally, any entire function).
Based on results in \cite{BS}, the following lemma shows that
the error incurred by this approximation decays exponentially in $n$, and provides
all relevant constants explicitly, allowing for a rigorous implementation.

\begin{lemma} \label{lem:approx_error}
Let $R > 1$, and let $\psi$ be a holomorphic function on a scaled Bernstein ellipse $E_{\gamma, R}$. Then for any $\rho \in (1, R)$ and $n \in \mathbb{N}$, the approximation error under Lagrange-Chebyshev interpolation obeys
$$
\sup_{x\in I} |(\psi - P_{n, \gamma}\psi)(x)| \leq
c_{\rho, R} \cdot \frac{\cosh(n \log \rho)}{\sinh(n \log R)}
\cdot \sup_{x\in I}| (\psi \circ \gamma \circ \sigma)(Re^{2\pi i x})|,
$$
where $c_{\rho, R} = \frac{\sinh(\log R)}{\cosh(\log R) - \cosh(\log \rho)}$ and
$\sigma(z) = (z+z^{-1})/2$.
\end{lemma}

\begin{proof}
Writing $\tilde{\psi} = \psi \circ \gamma$, we have
\[
\sup_{x\in I} |(\psi - P_{n, \gamma}\psi)(x)| =
\sup_{x \in [-1,1]}|(\tilde{\psi} - P_n \tilde{\psi})(x)
\leq \sup_{z \in E_\rho} |(\tilde{\psi} - P_{n, \gamma}\tilde{\psi})(z)|,
\]
for any $\rho \in (1, R)$. Using \cite[Lemma 2.6]{BS}, we obtain the bound
\[\sup_{z \in E_\rho} |(\tilde{\psi} - P_{n, \gamma}\tilde{\psi})(z)|
\leq c_{\rho, R} \cdot \frac{\cosh(n \log \rho)}{\sinh(n \log R)}
\cdot \sup_{z\in E_{R}}
|\tilde{\psi}(z)|, \]
with the constant
$c_{\rho, R} = \sinh(\log R) / (\cosh(\log R) - \cosh(\log \rho))$.
The assertion follows by application of the maximum modulus principle, and observing that $\sigma$ maps a circle of radius $R$ to the boundary of
$E_R$:
\begin{equation*}
\sup_{z\in E_{R}} |\tilde{\psi}(z)| \leq
\sup_{z\in \partial E_{R}} |\tilde{\psi}(z)| =
\sup_{x\in I}| (\tilde{\psi}\circ \sigma)(Re^{2\pi i x})|. \qedhere
\end{equation*}
\end{proof}

\begin{remark} \label{rem:ba}
Expressing a quantity as the extremum of a function over an interval
(as done in the previous lemma)
allows for rigorous computational bounding by using ball arithmetic,
as, for example, implemented in the Arb library \cite{arb}. In practice,
in order to obtain rigorous bounds on a function $f$
over the interval $I = [0, 1]$, say, we can
evaluate $f$ on $k \in \mathbb{N}$ equal-sized subintervals $I_l = [c_l-r, c_l+r]$,
$r = |I|/k, c_l = (l-1/2)\cdot r$, $l = 1, \ldots, k$.
The image intervals $J_l = [c'_l-r'_l, c'_l+r'_l]$ output by the
ball arithmetic calculation are then guaranteed to satisfy
$f(I_l) \subseteq J_l$, yielding effective lower and upper bounds on $f$.
\end{remark}

Combining \eqref{eq:derivative_estimate_2}, \eqref{eq:two_parts}
and Lemma \ref{lem:approx_error} yields the inequality
\begin{equation*}
\sup_{x \in I} \frac{\mathcal{L}_\epsilon f(x)}{f(x)}
\leq \frac{\mathcal{L}_\epsilon f(c)}{f(c)} +
\frac{r}{\inf_{x \in I} f(x)^2} \left( \delta_{n,\rho,R}(\psi) + \sup_{x \in I} |(P_{n, \gamma}\psi)(x)| \right),
\end{equation*}
where $\delta_{n,\rho,R}(\psi)$ denotes the upper bound from Lemma \ref{lem:approx_error}
which decays exponentially with $n$. We note that, as in Remark \ref{rem:ba},
all suprema and infima on the right-hand
side of this inequality can be rigorously bounded via ball arithmetic.

\begin{remark} \label{rem:sup_inf}
We note that with the same method as above, a rigorous \emph{lower} bound on
the \emph{infimum} of $(\mathcal{L}_\epsilon f) / f$ can be obtained, yielding
\begin{equation*}
\inf_{x \in I} \frac{\mathcal{L}_\epsilon f(x)}{f(x)}
\geq \frac{\mathcal{L}_\epsilon f(c)}{f(c)} -
\frac{r}{\inf_{x \in I} f(x)^2} \left( \delta_{n,\rho,R}(\psi) + \sup_{x \in I} |(P_{n, \gamma}\psi)(x)| \right).
\end{equation*}
\end{remark}

\subsection{Algorithmic implementation of rigorous bounds} \label{sec:algorithm_impl}

With the results of the previous sections, we are now ready to describe our
algorithm for computing rigorous and tight estimates on $(\mathcal{L}f) / f$
for a given transfer operator $\mathcal{L} = \mathcal{L}_\epsilon$, which
will form the key ingredient in the proofs of Theorems \ref{thm:entropy}-\ref{thm:hdim};
see Algorithm \ref{alg:compute_bounds}.

\begin{algorithm}
\caption{Rigorous lower/upper bounds on $(\mathcal{L}f) / f$}
\label{alg:compute_bounds}
\begin{algorithmic}
\Function{compute\_bounds}{$\mathcal{L}\colon g \mapsto \mathcal{L}g,
d\mathcal{L}\colon (g, g') \mapsto (\mathcal{L}g)', m, n, c, r$}
\State $\{e_l\}, \{e'_l\} \gets $ \Call{cheb\_basis}{$m, c, r$}
\Comment{basis functions and their derivatives}
\State $M \gets $ \Call{compute\_matrix}{$\mathcal{L}, \{e_l\}$}
\State $v \gets$ \Call{power\_method}{$M$}
\Comment{e.vec. corresponding to leading e.val.}
\State $f \gets \sum_l v_l\cdot e_l ~;~ f' \gets \sum_l v_l\cdot e'_l$
\State $f_{min} \gets$ \Call{inf\_ball}{$f$}
\State \Call{assert}{$f_{min} > 0$} \Comment{abort if eigenfunction non-positive}
\State $\mathcal{L}f \gets \mathcal{L}(f)~;~ (\mathcal{L}f)' \gets d\mathcal{L}(f, f')$
\State $\psi \gets (\mathcal{L}f)' \cdot f - \mathcal{L}f \cdot f'$
\State $P_{n,\gamma}(\psi), error(P_{n,\gamma}, \psi) \gets $ \Call{pn\_approximate}{$\psi, n$}
\State $S \gets $\Call{sup\_ball}{$|P_{n,\gamma}(\psi)|$} $+ error(P_{n,\gamma}, \psi) $
\State \Return $\mathcal{L}f(c) / f(c) \mp S \cdot r / f_{min}^2$
\EndFunction
\end{algorithmic}
\end{algorithm}

The algorithm requires access to mappings which,
for a given differentiable function $g$ and its derivative $g'$,
produce $\mathcal{L}g$ and $(\mathcal{L}g)'$.
Following Section \ref{sec:choosing_fg}, it begins by computing
a matrix representation $M$ in the Chebyshev basis
for a finite-rank approximation of $\mathcal{L}$ (\textsc{compute\_matrix})
and obtaining the eigenvector $v$ corresponding to its leading
eigenvalue via the power method (\textsc{power\_method}).

The corresponding eigenfunction $f$ and its derivative are
then used to define $\psi = (\mathcal{L}f)'\cdot f - \mathcal{L} \cdot f'$.
The algorithm proceeds by computing $P_{n,\gamma}(\psi)$ and the
upper bound on the approximation error $error(P_{n,\gamma}, \psi) = \sup |\psi - P_{n,\gamma}(\psi)|$
described in Lemma \ref{lem:approx_error} (\textsc{pn\_approximate}).
Finally, it returns the (rigorous) lower and upper bounds,
following the respective formulas derived in Section \ref{sec:rigorous_ineq}.

We note that all calculations are performed using ball arithmetic to account
for potential accumulation of errors. In particular, as in Remark \ref{rem:sup_inf},
all $\sup$ (\textsc{sup\_ball}) and $\inf$ (\textsc{inf\_ball}) computations
are performed by function evaluation on $k$ equal-sized subintervals of $[c-r, c+r]$.

\section{Proofs of results}\label{sec:proofs}
\subsection{Proof of Theorem \ref{thm:entropy}}

In order to complete the proof of Theorem \ref{thm:entropy}
we combine the rigorous computational estimation
presented in Sections \ref{sec:choosing_fg}-\ref{sec:algorithm_impl}
with Theorem \ref{thm:bounds_h}.

\begin{algorithm}
\caption{Rigorous lower/upper bounds on $P'(0) = - h(\mu)$}
\label{alg:compute_slope}
\begin{algorithmic}
\Function{compute\_pressure\_derivative}{$\epsilon, m, n, c, r$}
\State $\mathcal{L}_\epsilon, d\mathcal{L}_\epsilon \gets $ \Call{operators}{$\epsilon$};
$ \mathcal{L}_{-\epsilon}, d\mathcal{L}_{-\epsilon} \gets $ \Call{operators}{$-\epsilon$}
\State $b_{min}, b_{max} \gets $
\Call{compute\_bounds}{$\mathcal{L}_{-\epsilon}, d\mathcal{L}_{-\epsilon}, m, n, c, r$}
\State $a_{min}, a_{max} \gets $
\Call{compute\_bounds}{$\mathcal{L}_\epsilon, d\mathcal{L}_\epsilon, m, n, c, r$}
\State $\beta \gets \log(b_{max}) ~;~ \alpha \gets -\log(a_{max})$
\State \Return $\alpha/\epsilon, \beta/\epsilon$
\EndFunction
\end{algorithmic}
\end{algorithm}

The computational step is summarized in Algorithm \ref{alg:compute_slope}. Making use
of \textsc{compute\_bounds} (see Algorithm \ref{alg:compute_bounds}) to compute rigorous upper bounds
\begin{equation*}
\frac{\mathcal L_{\epsilon} g(x)}{g(x)} \leq e^{-\alpha} \qquad \text{ and } \qquad
\frac{\mathcal L_{-\epsilon} f(x)}{f(x)} \leq e^{\beta},
\end{equation*}
bounds on $h(\mu)$ are obtained via Theorem \ref{thm:bounds_h} as
$\alpha/\epsilon \leq h(\mu) \leq \beta/\epsilon$.
In order to obtain the given precision, we make the following set of parameter choices in executing
the algorithm:

\begin{enumerate}[(a)]
\item We take the value $\epsilon = 10^{-50}$.
\item We choose $m = 160$ for the rank of the approximation of $\mathcal{L}_\epsilon$.
\item We choose $n = 200$ for the rank of the Lagrange-Chebyshev approximation of the function $\psi$ in Section \ref{sec:rigorous_ineq}.
\item For estimating suprema and infima with ball arithmetic as described in Section \ref{sec:rigorous_ineq}, we use $k = 250$ intervals.
\item We use the value\footnote{We note that $\psi$ is holomorphic on
$E_{\gamma,R}$ for any $1 < R < \exp(\operatorname{arccosh}(3))\approx 5.8$.}
$R = 5.5$ and $\rho = 1.001$ for computing the bound on the approximation error in Lemma \ref{lem:approx_error}.
\end{enumerate}
The resulting values are
\begin{align*}
\alpha/\epsilon = 1.0563130740\,7297055209\,9568877064\,0651679335\,4262184005\,{\color{gray}66}\ldots \\
\beta/\epsilon  = 1.0563130740\,7297055209\,9568877064\,0651679335\,4262184005\,{\color{gray}75}\ldots
\end{align*}
yielding the asserted value for $h(\mu)$.

\begin{remark}
We note that in \eqref{eq:two_parts}, the approximation error incurred by $P_{n,\gamma}$
is negligible compared to the actual size of the approximated supremum.
With our choices of parameters (in particular, approximation rank $n = 200$),
the error is
$\sup_{x \in I} | (\psi - P_{n,\gamma}\psi)(x)  | < 2 \cdot 10^{-142}$,
whereas $\sup_{x \in I} | \psi (x) | \approx \sup_{x \in I} | (P_{n,\gamma} \psi) (x) | \approx  10^{-112}$.
\end{remark}

\subsection{Proof of Theorem \ref{thm:freqs}}

By a simple application of the Birkhoff ergodic theorem, 
the frequency $f_i$ ($i=1,2,3$) of each of the three digits can be expressed in terms of the $T$-invariant measure $\mu$ as
$f_i = \mu(I_i)$, with
$$
I_1 = \left[0, \sqrt{2}-1\right), ~
I_2 = \left[\sqrt{2}-1, \sqrt{3}-1\right), ~
I_3 = \left[ \sqrt{3}-1,1\right).
$$
Thus we need to estimate the respective measure for each of these digits.   To this end we want to replace the family of operators $\mathcal L_t$ by the following.

\begin{definition}
Given $t\in \mathbb R$, we
define operators
$\mathcal N_{i,t} \colon C([0,1]) \to C([0,1])$ for $i=1,2,3$ by
$$
\mathcal N_{i,t} f(x) =
\mathcal L_{t} f(x) + (e^t-1)
|T_i'(x)| f(T_i).
$$

\end{definition}

We can denote by $e^{R_i(t)}$ the spectral radius of the operator $\mathcal N_{i,t}$.

\begin{definition}
For $i=1,2,3$ we define  $R_i\colon \mathbb  R \to \mathbb  R$ as
$$
R_i(t) =  \limsup_{n\to +\infty} \frac{1}{n} \log \|\mathcal N_{i,t} ^n1\|_\infty,
$$
where $1(x) = 1$ is the constant function and $\| f \|_\infty = \sup_{x\in [0, 1]} |f(x)|$.
\end{definition}

The following results are analogues of Lemma \ref{lem:pressure} and Corollary \ref{cor:entropy_bounds}.

\begin{lemma} For $i=1,2,3$ the function $R_i\colon \mathbb R \to \mathbb R$ has the following properties:
\begin{enumerate}
\item $R_i(0) = 0$.
\item 
$R_i$ is $C^\infty$ (even real analytic), monotone decreasing and convex.
\item $\frac{d}{dt}R_i(t)|_{t=0} = - f_i$.
\end{enumerate}
\end{lemma}

\begin{corollary}\label{cor:r_bounds}
For $\epsilon > 0$ and $i=1,2,3$ we can bound
$$
\frac{R_i(-\epsilon)}{\epsilon} \geq \mu(I_i)\geq \frac{|R_i(\epsilon)|}{\epsilon}.
$$

\end{corollary}

In order to convert these bounds into practical estimates we need the analogue of
Lemma \ref{lem:pe_bound}, now for the operators $\mathcal N_{i, t}$ ($i=1,2,3$).

\begin{lemma}
Let $t, \alpha \in \mathbb{R}$, $i=1, 2, 3$. If there exists a positive continuous
function $g\colon[0,1]  \to \mathbb R^+$ such that
$\sup_x \frac{\mathcal N_{i, t} g(x)}{g(x)} \leq  e^\alpha$, then $R_i(t) \leq \alpha$.
\end{lemma}

The remainder of the proof of the theorem is completely analogous to that of Theorem \ref{thm:entropy},
where for the computational step we also use the exact same parameter values.

\subsection{Proof of Theorem \ref{thm:hdim}}

We can adapt the method of the proof of Theorem \ref{thm:entropy} to prove Theorem \ref{thm:hdim}.
To begin we need to replace the operators $\mathcal L_t$ by the following operators.

\begin{definition}
For any $t \in \mathbb R$, let
 $\mathcal M_t\colon C([0,1])\to C([0,1])$ be the bounded linear operator defined by
$$
\mathcal M_t f(x) = |T_0'(x)|^t f(T_0x) + |T_2'(x)|^t f(T_2x) \quad \forall f \in C([0,1]), x \in [0,1].
 $$
 \end{definition}
Denoting the spectral radius of $\mathcal M_t$ by $e^{Q(t)}$,
the following re-formulation follows from the  spectral radius theorem.

\begin{definition}
We can write  $Q \colon \mathbb R  \to \mathbb  R$ as
$$
Q(t) =  \lim_{n\to +\infty} \frac{1}{n} \log \|\mathcal M_t ^n{\mathbf 1}\|_\infty,
$$
where ${\mathbf 1}(x) = 1$ is the constant function.
\end{definition}

The following result is the analogue of Lemma \ref{lem:pressure}.

\begin{lemma} The function $Q \colon \mathbb R \to \mathbb R$
 is $C^\infty$ (even real analytic), monotone decreasing and convex.
\end{lemma}

The following lemma can be proved analogously to Lemma \ref{lem:pe_bound}.

\begin{lemma}\label{lem:mt_bounds}
Let $t > 0, \alpha \in \mathbb R$, and let $g \colon [0, 1] \to \mathbb R^+$ be a positive continous function.
Then the following implications hold:
\begin{align}
\inf_{x \in [0, 1]} \frac{\mathcal M_t g(x)}{g(x)} \geq e^\alpha
&\implies Q(t) \geq \alpha,  \\
\sup_{x \in [0, 1]} \frac{\mathcal M_t g(x)}{g(x)} \leq e^\alpha
&\implies Q(t) \leq  \alpha.
\end{align}
\end{lemma}

The connection between $\dim(\mathcal E)$ and the function $Q$ is the following classical result \cite{bowen, ruelle1}.

 \begin{lemma}[Bowen's Theorem]\label{lem:bowen}
 The Hausdorff{} dimension of  $\mathcal E$ corresponds to the solution $t = \dim_H(\mathcal E)$
 to $Q(t) = 0$.
\end{lemma}

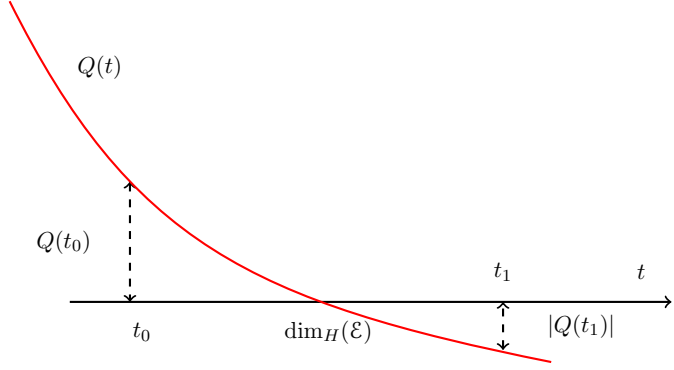
\begin{figure}
\begin{center}
\begin{tikzpicture}[thick,scale=0.8, every node/.style={scale=0.8}]
\draw[->] (0,0) --(10,0);
\node at (1.2,  -0.5) {$t_0$};
\node at (7.2,  0.5) {$t_1$};
\draw[dashed, <->] (1,2.0)--(1,0);
\draw[dashed, <->] (7.2,-0.8)--(7.2,0);
\node at (-0.1,  1.0) {$Q(t_0) $};
\node at (8.5,  -0.4) {$|Q(t_1)| $};
\draw[red](-1,5) .. controls (1,1)  and (3,0) .. (8,-1.0);
\node at (4.3,  -0.5) {$\dim_H(\mathcal E)$};
\node at (0.5,  3.9) {$Q(t)$};
\node at (9.5,  0.5) {$t$};
\end{tikzpicture}
\end{center}
\caption{$t_0 \leq \dim_H(\mathcal E) \leq t_1$}
\label{fig:dim_e}
\end{figure}

Lemmas \ref{lem:mt_bounds} and \ref{lem:bowen} and the monotonicity of $Q(t)$ immediately yield the following corollary (see Figure \ref{fig:dim_e}).

\begin{corollary}\label{cor:dimh}
Let $t_1 > t_0 > 0$ and assume there are positive continuous functions
$f,g \colon [0, 1] \to \mathbb R^+$ and $\alpha, \beta > 0$ such that
$$\inf_x \frac{\mathcal M_{t_0} g(x)}{g(x)} \geq e^\beta \quad \text{ and }
\quad \sup_x \frac{\mathcal M_{t_1} f(x)}{f(x)} \leq e^{-\alpha}.$$
Then $t_0 < \dim_H(\mathcal E) < t_1$.
\end{corollary}

\begin{remark}
It is worth noting that in the proofs of Theorems \ref{thm:entropy} and \ref{thm:freqs} our goal
was to estimate the \emph{derivative} of the pressure function, which required obtaining two upper bounds
on expressions of the form $(\mathcal L f) / f$. Here, on the other hand, we are aiming to
accurately estimate the \emph{zero} of such function, which is achieved by computing a pair of lower and upper bounds.
\end{remark}

Corollary \ref{cor:dimh} provides us with a criterion that we can easily rigorously
verify with the computational method described in Section \ref{sec:method},
in particular using Algorithm \ref{alg:compute_bounds}.

\begin{algorithm}
\caption{Rigorous lower/upper bounds on zero of $Q$}
\label{alg:compute_zero}
\begin{algorithmic}
\Function{compute\_pressure\_zero}{$t_0, t_1, m, n, c, r, \delta$}
\While{$t_1 - t_0 > \delta$}  \Comment{iterate to desired precision $\delta$}
\State $t \gets (t_0 + t_1)/2$
\State $\mathcal{M}_t, d\mathcal{M}_t \gets $ \Call{operators}{$t$}
\State $\alpha_{min}, \alpha_{max} \gets $
\Call{compute\_bounds}{$\mathcal{M}_t, d\mathcal{M}_t, m, n, c, r$}
\If{$\alpha_{min} > 0$}
\State $t_0 \gets t$
\ElsIf{$\alpha_{max} < 0$}
\State $t_1\gets t$
\Else \Comment{also invoke this case if \textsc{compute\_bounds}
fails ($f_{min} \leq 0$)}
\State $m \gets m+1$
\EndIf
\EndWhile
\State \Return $t_0, t_1$
\EndFunction
\end{algorithmic}
\end{algorithm}

The entire procedure is described in Algorithm \ref{alg:compute_zero}.
It is based on a binary search principle:
beginning with a small initial approximation rank $m$ and an interval $[t_0, t_1]$
guaranteed to contain $\dim_H(\mathcal E)$, we compute rigorous bounds on
$\inf$ and $\sup$ of $\mathcal M_t f / f$ for $t$ the mid-point of $[t_0, t_1]$ (\textsc{compute\_bounds}).
If $\inf (\mathcal M_t f) / f > 1$, then $\dim_H(\mathcal E) > t$ and the search continues
on the sub-interval $[t, t_1]$; if $\sup (\mathcal M_t f) / f < 1$,
then $\dim_H(\mathcal E) \leq t$ and the search continues on $[t_0, t]$.
If neither of the conditions hold, or if the approximate eigenfunction $f$
computed in \textsc{compute\_bounds} fails to be positive (both of which can occur due to
insufficiently accurate finite-rank approximation of the operator),
then we increment the approximation rank $m$ and repeat the procedure on $[t_0, t_1]$.
This iteration proceeds until the required accuracy $\delta > 0$ is reached, that is, until
$t_1 - t_0 < \delta$.

We perform this computation with the initial values
$t_0 = 0.5$, $t_1 = 0.8$ and $m = 10$, and the parameters $n = 100$, $k = 250$, $\rho = 1.001$ and $R=5.5$,
which yields the asserted value and accuracy for
$\dim_H(\mathcal E)$, proving Theorem \ref{thm:hdim}.
We note that in practice the algorithm reaches this accuracy after incrementing
the approximation rank for $\mathcal M_t$ to $m = 75$.

\section{Comparison with other approaches} \label{sec:comparison}

In this section we compare the estimate in Theorem \ref{thm:entropy} with the results obtained using other methods.

\subsection{The finite section method}

There is a more classical approach to estimating the entropy of the Bolyai-R\'enyi map directly.
We have seen that the metric  entropy $h(\mu)$ (or equivalently the Lyapunov exponent)
can be described using the maximal eigenvalue and the associated eigenfunction
(and eigenmeasure) of a transfer operator, see \eqref{eq:pressure_der}. We write
\[h(\mu) = \mu(\log |T'|) = \ell^*(\eta \cdot \rho), \]
where $\rho, \ell^*$ are the eigenfunction and eigenfunctional of $\mathcal{L}_0$ corresponding to the leading eigenvalue $1$, normalized so that $\ell^*(f) = 1$,
and $\eta(x) = \log(|T'(x)|)$ for $x\in I=[0,1]$.

Using the Lagrange-Chebyshev interpolation method (Definition \ref{defn:lagcheb}), we approximate
$\mathcal L_0$ by a rank $m$ operator (for some $m \in \mathbb{N}$), represented by an $m\times m$ matrix $M$.
The right and left eigenvectors $v, w \in \mathbb R^m$ corresponding to the leading eigenvalue of $M$ give rise to the corresponding eigenfunction $\rho_m$ and the eigenfunctionals $\ell_m$ of the approximating rank $m$ operator via
$$
\rho_m(x) = \sum_{l=0}^{m-1} v_l e_l(x) \quad
\text{ and } \quad
\ell_m(f) = \sum_{l=0}^{m-1} w_l e_l^*(f),
$$
where $e_m$ and $e^*_m$ are as in Definition~\ref{defn:lagcheb}, up to an (affine) change of coordinates, taking into
account the (non-standard) interval $I=[0,1]$.
Writing $b_l = e_l^*(\eta \cdot \rho_m)$, the approximation to the entropy can now be computed as
$$
h_m = \frac{\ell_m(\eta \cdot \rho_m)}{\ell_m(\rho_m)}
=     \frac{\sum_{l=0}^{m-1} b_l w_l}{\sum_{l=0}^{m-1}v_l w_l}.
$$
The denominator in this expression is needed for normalization.

\begin{example}
If we let $m=100$ and carry out the above calculation we have the following {\it heuristic} estimate for the entropy
\begin{align*}
h_{100} = 1.&0563130740\,7297055209\,9568877064\,0651679335\,4262184005\, \\
            &7092244740\,0283696700\,9505655203\,1501166170\,438688675 \ldots
\end{align*}
\end{example}
\noindent
This agrees with the value in Theorem \ref{thm:entropy}. 
Theorem 3.3 and Corollary 3 of \cite{BS} guarantee exponential convergence (in $m$)
of $h_m$ to $h(\mu)$, where the convergence rate can be bounded using the complex contraction ratios of the inverse branches of $T$ on suitable Bernstein ellipses. However the implied constant depends on the resolvent
of the operator, 
making rigorous error estimates more difficult to obtain in this case.

\subsection{The periodic point method}
The approach in \cite{jp} was based on using the data $(T^n)'(x)$ for fixed points $T^nx=x$ for $T^n\colon[0,1] \to [0,1]$
($n \geq 1$).  More precisely, one considers a  determinant function of two variables formally defined by
$$
d(z,s) = \exp \left(
- \sum_{n=1}^\infty \frac{z^n}{n} \sum_{T^nx=x} \frac{|(T^n)'(x)|^{-s}}{1-1/(T^n)'(x)} \right),
$$
which converges for any $z \in \mathbb C$  and  $\operatorname{Re}(s) $ sufficiently large, and has an analytic extension to $\mathbb C^2$.    The entropy can then be written in the form
$$
h(\mu) =\frac{\partial d(z,s)}{ds}\big/\frac{\partial d(z,s)}{dz}\big\vert_{s=0, z=1}.
$$
In order to convert this into a useful algorithm  we truncate the Taylor series  for $d(z,s)$ in $z$.
 In particular, for any  $\frac{1}{3-\sqrt{3}} < \theta  < 1$ and $M \geq 1$, we have
$$
d(z,s) = 1 +  \sum_{n=1}^M a_n(s) z^n + O(\theta^{M^2}),
$$
where $a_n(s)$ depends only on the weights associated to periodic points of period at most $n$.

\begin{example}
With the choice $M=12$ the entropy was rigorously estimated
to be\footnote{Here the notation $x = A \pm 10^{-\kappa}$ signifies
that $A - 10^{-\kappa} \leq x \leq A +10^{-\kappa} $ holds rigorously.}
$h(\mu) = 1.05631307402 \pm 4.1 \times  10^{-6}$ (see \cite[\S 5]{jp})
and heuristically estimated to $9$ decimal places as $1.056313074$.
This required computing  the $3^M=531441$ periodic points $T^nx=x$
and their weights $(T^n)'(x)$
for $1 \leq n \leq 12$.
\end{example}

Although the error terms can be bounded with comparative ease, an issue for this method is the exponential growth in data that needs to be managed as $M$ increases.

\section{Final remarks}
\label{sec:remarks}

In this last section, we briefly comment on some directions in which our method and results generalize, as
well as on some of their limitations.

\begin{remark}
In the interest of clarity, we have chosen to define the operators
$\mathcal L_t$, $\mathcal M_t$ and $\mathcal N_{j,t}$ separately.
Instead, we could have introduced a single more general class of operators of the form
$$
\widehat {\mathcal L}_t f(x) =  \sum_{i=0}^2|T_i'(x)| e^{g(T_ix)} f(T_ix),
$$
with different choices of $g \in C^1([0, 1])$ giving rise to the three operators of interest.
\end{remark}

\begin{remark}
Our method readily applies to generalizations of the Bolyai-R\'enyi transformation \cite{schweiger}.
As an example, for $m \in \mathbb{N}\setminus\{1\}$ we consider $\tilde T_m \colon [0, 1) \to [0, 1)$
given by
\[
 \tilde T_m(x) = (x+1)^m - 1,
\]
noting that for $m = 2$ this reduces to the classical Bolyai-R\'enyi map $T = \tilde T_2$.
For each $m$, the map $\tilde T_m$ has exactly $2^m-1$ contractive inverse branches given by
$x \mapsto (x+i)^{1/m} - 1$, $i = 1, 2, \ldots, 2^m-1$. Completely analogously to Theorem \ref{thm:entropy},
we can rigorously estimate the entropy $h(\tilde T_m, \mu_m)$ (where $\mu_m$ denotes the respective $\tilde T_m$-invariant
probability measure absolutely continuous to Lebesgue measure).
For $m = 3, 4, \ldots, 10$ we obtain the following values,
each accurate to the number of decimal places presented:
\begin{table}[ht]
\centering
\begin{tabular}{c|c}
m & $h(\tilde T_m, \mu_m)$ \\
\hline
3  & 1.83495\,44938\,47482\ldots \\
4  & 2.50156\,90070\,03226\ldots \\
5  & 3.10685\,89449\,66953\ldots \\
6  & 3.67309\,55489\,06997\ldots \\
7  & 4.21213\,20147\,18818\ldots \\
8  & 4.73108\,63064\,39220\ldots \\
9  & 5.23459\,49760\,98698\ldots \\
10 & 5.72585\,67503\,35337\ldots
\end{tabular}
\end{table}
\end{remark}

\begin{remark}
We have presented our results in the specific setting of expansions by
iterated radicals. This is in part to illustrate the method, and as a paradigm of a more
general method, which originated in \cite{PV1}, while also
giving better results on the map $T$, interesting in its own right.
The important properties that we need (and which are easy to check  for $T$) are the following:
\begin{enumerate}
\item $T$ is conformal (which is automatic since it is one dimensional);
\item $T$ is $C^1$; 
\item $T$ has a finite number of inverse branches;
\item $T$ is expanding, i.e., there exists $\kappa > 1$ such that  $\inf_x |T'(x)| \geq \kappa$;
\item $T$ is Markov (even Bernoulli, in this case).
\end{enumerate}
This fits naturally in the setting of $f$-expansions \cite{renyi}.{}
These properties allow us to make of use Theorem~\ref{thm:bounds_h},
Corollaries~\ref{cor:r_bounds} and \ref{cor:dimh}. In order to bound the relevant
quantities rigorously to a very high precision using the algorithm presented, we crucially used the fact
that $T$ is analytic, ensuring exponential convergence of the Lagrange-Chebyshev interpolation scheme \cite{BS}.
\end{remark}

\begin{remark}
While as above our method generalizes beyond the exact setting of this paper,
we should also acknowledge its limitations. In particular, the application of
our method presented here is specifically tailored
for estimating quantities relating to the top eigenvalue of a transfer operator.
It is not directly applicable to estimating quantities such as the rate of mixing of a transformation,
which is determined by the transfer operator's spectral gap
and has been previously estimated with the periodic point method \cite{jp}.
\end{remark}

\begin{remark}
A real number is called \emph{computable} if there is an algorithm (corresponding to a Turing machine)
computing a value abitrarily close to that number.
The set of computable numbers is countable.
It is a consequence of the algorithm presented that the entropy $h(\mu)$ is a computable number.
\end{remark}

\begin{remark}
Our algorithm can be used to obtain any specified accuracy on the estimated quantities
(given sufficient computational resources).
It is natural to ask how its time complexity grows with an increase in desired accuracy.
This boils down to the time complexity of \textsc{compute\_bounds} (Algorithm \ref{alg:compute_bounds}),
which is driven by computing the top eigenvalue and corresponding eigenvector of an $m \times m$ matrix,
and evaluating the approximation $P_n \psi$ on $k$ intervals.
The first of these has a complexity of $O(p \cdot m^2)$,
with $p$ the number of steps of the power method (both $p$ and $m$ can be expected to be $O(\log(1/\delta))$ for a desired
eigenvector accuracy of $\delta$), the second has a complexity of $O(k \cdot n)$.
\end{remark}

\Addresses

\end{document}